\documentclass[10pt]{article}
\usepackage[margin= 1.2in]{geometry}

\usepackage{amsmath}
\usepackage{amssymb}
\usepackage{amsthm}
\usepackage{dsfont}
\usepackage[T1]{fontenc}
\usepackage[mathscr]{eucal}
\usepackage[usenames,dvipsnames]{color}
\usepackage{graphicx}
\usepackage{mathrsfs}
\usepackage{nameref}
\usepackage{pgfplots}
\usepackage{setspace}
\usepackage{textcomp}
\usepackage{tikz}
	\usetikzlibrary{matrix}
\usepackage[normalem]{ulem}
\usepackage{verbatim}

\numberwithin{equation}{section}

\usepackage{stmaryrd}

\title{\Large Anticommutation in the Presentations of Theta-Deformed Spheres}
\author{Benjamin Passer\thanks{Partially supported by National Science Foundation Grants DMS 1300280 and DMS 1363250} \\ Washington University\\ St. Louis, MO 63130 \\ \\bpasser@math.wustl.edu}

\begin{document}


\newcommand{\textapprox}{\raisebox{0.5ex}{\texttildelow}}
\def\polhk#1{\setbox0=\hbox{#1}{\ooalign{\hidewidth\lower1.5ex\hbox{`}\hidewidth\crcr\unhbox0}}}

\def\ba{\begin{aligned}}
\def\ea{\end{aligned}}
\def\be{\begin{equation}}
\def\ee{\end{equation}}
\def\beu{\begin{equation*}}
\def\eeu{\end{equation*}}

\def\C{\mathbb{C}}
\def\D{\mathbb{D}}
\def\R{\mathbb{R}}
\def\Rn{\mathbb{R}^n}
\def\S{\mathbb{S}}
\def\Sn{\mathbb{S}^n}
\def\Z{\mathbb{Z}}
\def\T{\mathbb{T}}
\def\N{\mathbb{N}}
\def\RP{\mathbb{RP}}
\def\Q{\mathbb{Q}}
\def\B{\mathbb{B}}

\def\fR{\mathfrak{R}}

\def\F{\mathfrak{F}}
\def\del{\partial}

\def\!={\neq}
\def\l{\langle}
\def\r{\rangle}

\theoremstyle{definition}

\newtheorem{defn}[equation]{Definition}
\newtheorem{lem}[equation]{Lemma}
\newtheorem{prop}[equation]{Proposition}
\newtheorem{thm}[equation]{Theorem}
\newtheorem{claim}[equation]{Claim}
\newtheorem{ques}[equation]{Question}
\newtheorem{fact}[equation]{Fact}
\newtheorem{axiom}[equation]{Technical Axiom}
\newtheorem{newaxiom}[equation]{New Technical Axiom}
\newtheorem{cor}[equation]{Corollary}
\newtheorem{exam}[equation]{Example}
\newtheorem{conj}[equation]{Conjecture}

\bibliographystyle{plain}
\maketitle

\begin{abstract}
We consider an analogue of the $\theta$-deformed even spheres, modifying the relations demanded of the self-adjoint generator $x$ in the usual presentation. In this analogue, $x$ is given anticommutation relations with all of the other generators, as opposed to being central. Using even $K$-theory, we show that noncommutative Borsuk-Ulam theorems hold for these algebras.
\end{abstract}

\vspace{.2 in}


\vspace{.2 in}


\vspace{.2 in}

\section{Introduction}\label{sec:intro}

When a Rieffel deformation procedure (\cite{ri93}) is applied to the function algebra $C(\S^k)$ of a sphere $\S^k$, the resulting $C^*$-algebra admits a succinct presentation (\cite{na97}, \cite{co01}). 

\begin{defn}
Let $\rho$ be an $n \times n$ matrix with $1$ on the diagonal, all entries unimodular, and $\rho_{jk} = \overline{\rho_{kj}}$. Then $C(\S^{2n-1}_\rho)$ and $C(\S^{2n}_\rho)$ are given by the following $C^*$-presentations.
\beu \ba
C(\S^{2n-1}_\rho) \cong C^*(z_1, \ldots, z_n \hspace{.1 in} &| \hspace{.1 in} z_jz_j^* = z_j^*z_j, \hspace{.1 in} z_kz_j = \rho_{jk} z_j z_k, \hspace{.1 in} z_1z_1^* + \ldots + z_nz_n^* = 1) \\
\\
C(\S^{2n}_\rho) \cong C^*(z_1, \ldots, z_n, x \hspace{.1 in} &| \hspace{.1 in} z_jz_j^* = z_j^*z_j, \hspace{.1 in} x = x^*, \hspace{.1 in} z_kz_j = \rho_{jk} z_j z_k, \hspace{.1 in} xz_j = z_jx, \hspace{.1 in} 
\\ &\phantom{|} \hspace{.1 in} z_1z_1^* + \ldots + z_nz_n^* + x^2 = 1)
\ea \eeu
\end{defn}

The presentation and notation illustrate the mentality suggested by the Gelfand-Naimark theorem, in that we view $C(\S^k_\rho)$ as a (noncommutative) function algebra with coordinate functions labeled $z_j$ or $x$. The noncommutativity relations of $C(\S^{k}_\rho)$ vary in a continuous fashion, and $C(\S^{2n}_\rho)$ is realized as a quotient $C(\S^{2n+1}_\omega) / \langle z_{n+1} - z_{n+1}^* \rangle$, where $\omega$ is such that $z_{n+1}$ is central and $\rho$ is the upper left $n \times n$ submatrix of $\omega$. However, this type of quotient is reasonable (that is, nondegenerate) even when $z_{n+1}$ anticommutes with some of the other $z_j$. We consider below the quotients $C(\S^{2n+1}_\omega)/\langle z_{n+1} - z_{n+1}^* \rangle$ when anticommutation always occurs.

\begin{defn}
Let $\rho$ be an $n \times n$ matrix with $1$ on the diagonal, all entries unimodular, and $\rho_{jk} = \overline{\rho_{kj}}$. Then $\fR^{2n}_\rho$ is defined by the following presentation.
\be\label{eq:antipres} \ba
\fR^{2n}_\rho \cong C^*(z_1, \ldots, z_n, x \hspace{.1 in} | \hspace{.1 in} &z_jz_j^* = z_j^*z_j, \hspace{.1 in} x = x^*, \hspace{.1 in} z_kz_j = \rho_{jk} z_j z_k, \hspace{.1 in} xz_j = -z_jx, \hspace{.1 in} 
\\  &z_1z_1^* + \ldots + z_nz_n^* + x^2 = 1)
\ea \ee
We adopt the convention that $\fR^{2n}$ without a subscript denotes $\fR^{2n}_\rho$ for $\rho$ a matrix of all ones. In this algebra, $z_1, \ldots, z_n$ commute with each other and anticommute with $x$.
\end{defn}

Each $\fR^{2n}_\rho$ may be realized as a noncommutative unreduced suspension (in the sense of \cite{pa15b}, Definition 3.4) of the unital $C^*$-algebra $C(\S^{2n-1}_\rho)$ given by the antipodal map, which is the order two homomorphism that negates each generator of the standard presentation. In general, if $\beta$ generates a $\Z_2$ action on a unital $C^*$-algebra $A$, then
\beu
\Sigma^\beta A := \{f \in C([0, 1], A \rtimes_\beta \Z_2): f(0) \in A, f(1) \in C^*(\Z_2) \}
\eeu
defines the noncommutative unreduced suspension of $(A, \beta, \Z_2)$. When $\beta$ is trivial, $\Sigma^\beta A$ is isomorphic to the unreduced suspension $\Sigma A$; this was considered in \cite{da15} and \cite{pa15b} in pursuit of noncommutative Borsuk-Ulam theory.

\begin{thm}[Borsuk-Ulam]
No continuous, odd maps exist from $\S^k$ to $\S^{k-1}$, and every continuous, odd map on $\S^{k-1}$ is homotopically nontrivial.
\end{thm}

Noncommutative Borsuk-Ulam theorems (\cite{ta12}, \cite{ya13}, \cite{ba15}, \cite{pa15a}, \cite{da15}, \cite{pa15b}, \cite{ch16}) arise when a $C^*$-algebraic theorem generalizes the result of passing the traditional Borsuk-Ulam theorem, or one of its various restatements and extensions, through Gelfand-Naimark duality. In short, these theorems describe when an equivariant map between two related $C^*$-algebras cannot exist, or when an equivariant self-map must behave nontrivially with respect to some invariant. The $\theta$-deformed odd spheres were considered for this purpose in \cite{pa15a} with $\Z_k$ rotation actions; the $\mathbb{Z}_2$ case is repeated below.

\begin{thm}\label{thm:oddBU}
Let $\alpha$ denote the antipodal $\Z_2$-action on any $C(\S^{k}_\rho)$, which negates each generator of the standard presentation. Then any $\alpha$-equivariant, unital $*$-homomorphism from $C(\S^{2n-1}_\rho)$ to $C(\S^{2n-1}_\omega)$ must induce a nontrivial map on $K_1 \cong \Z$. Consequently, for any $k$, there are no $\alpha$-equivariant, unital $*$-homomorphisms from $C(\S^k_\gamma)$ to $C(\S^{k+1}_\delta)$.
\end{thm}

The $K$-theory of $\theta$-deformed spheres exactly matches that of the commutative case, as Rieffel deformation preserves $K$-theory (\cite{ri93b}), but there are more detailed descriptions (\cite{na97}, \cite{pe13}). 
\beu
\begin{array}{cc} K_0(C(\S^{2n-1}_\rho)) \cong \Z & K_1(C(\S^{2n-1}_\rho)) \cong \Z \\ \\ K_0(C(\S^{2n}_\rho)) \cong \Z \oplus \Z & K_1(C(\S^{2n}_\rho)) \cong \{0\} \end{array}
\eeu
On the other hand, the anticommutation present in $\fR^{2n}_\rho$ is not a matter that can be resolved continuously, as the generator $x$ is self-adjoint. Despite this obstruction, a noncommutative Borsuk-Ulam theorem does hold.

\begin{thm}
Let $A_{2n}, B_{2n}$ be $C^*$-algebras of the type $\fR^{2n}_\rho$ or $C(\S^{2n}_\rho)$ (the types may differ or use different matrices $\rho$), both of which are equipped with an antipodal map that negates each generator. Both algebras have $K_0 \cong \Z \oplus \Z$, with the first summand generated by the trivial projection $1$. Any equivariant, unital $*$-homomorphism from $A_{2n}$ to $B_{2n}$ must induce a $K_0$ map that is nontrivial on the second summand.
\end{thm}

\section{Commutation}

As seen in \cite{da15}, it is known that the even $\theta$-deformed sphere $C(\S^{2n}_\rho)$ may be written as the unreduced suspension of $C(\S^{2n-1}_\rho)$.
\be\label{eq:even=oddsuspension}
C(\S^{2n}_\rho) \cong \Sigma C(\S^{2n-1}_\rho) \cong \{f \in C([-1, 1], C(\S^{2n-1}_\rho)): f(-1) \in \C, f(1) \in \C\}
\ee
One approach is to define a homomorphism $C(\S^{2n}_\rho) \to \Sigma C(\S^{2n-1}_\rho)$ using the presentation of $C(\S^{2n}_\rho)$, and then to see that the realization of $C(\S^{2n-1}_\rho)$ as an algebra of functions from the positive section of $\S^{n-1}$ to the quantum torus $C(\T^n_\rho)$ (\cite{na97}) helps produce an inverse map. Further, when the odd sphere $C(\S^{2n+1}_\omega)$ is such that $C(\S^{2n+1}_\omega)/\langle z_{n+1} - z_{n+1}^* \rangle \cong C(\S^{2n}_\rho)$, $C(\S^{2n+1}_\omega)$ itself is isomorphic to a similar algebra of functions on a disk, instead of an interval. 
\be\label{eq:odd=diskfunctions}
\omega = \left[\begin{array}{ccc} \rho & & 1 \\ & & \vdots\\ 1 & \ldots & 1 \end{array} \right] \implies C(\S^{2n+1}_\omega) \cong \{g \in C(\overline{\D}, C(\S^{2n-1}_\rho)): g|_{\S^1} \textrm{ is } \C\textrm{-valued}\}
\ee
The antipodal maps for these algebras of functions, which are indexed below by dimension, satisfy the following identities.
\beu \ba
\alpha_{2n}(f)[x] = \alpha_{2n-1}(f(-x)), x \in [-1, 1]
\ea \eeu
\beu \ba
\alpha_{2n+1}(g)[u] = \alpha_{2n-1}(g(-u)), u \in [-1,1]
\ea \eeu
To prove noncommutative Borsuk-Ulam theorems for even spheres, we note that the odd dimensional case relied on the following idea from \cite{pa15a}, the backbone of Theorem \ref{thm:oddBU}.

\begin{lem}\label{lem:oddevenK_1}
If $U \in \mathcal{U}(C(\S^{2n-1}_\rho))$ is a unitary matrix with even entries, then in $K_1(C(\S^{2n-1}_\rho)) \cong \Z$, $[U]$ is an even integer. Because there is a dimension $2^{n-1}$ unitary matrix $Z_\rho(n)$ with odd entries and $[Z_\rho(n)] = 1$, it follows that any unitary matrix $W$ of dimension $(2m+1)2^{n-1}$ with odd entries must have $[W] \in 2 \Z + 1$.
\end{lem}

The lemma showcases the particularly nice interaction between $K_1$ and the antipodal action for an odd sphere. The $K_1(C(\S^{2n-1}_\rho))$ generator $Z_\rho(n)$ has $*$-monomial entries (see \cite{na97}), which are all odd. Before considering anticommutation relations, we seek to establish a related identity on even spheres. Motivated by \cite{da02}, \cite{co01}, and \cite{pe13}, we see that in $K_0(C(\S^{2n}_\rho)) \cong \Z \oplus \Z$, there is a distinguished projection
\beu
P_\rho(n) = \left[\begin{array}{cc} \frac{1+x}{2}I_{2^{n-1}} & Z_\rho(n) \\ Z_\rho(n)^* & \frac{1-x}{2}I_{2^{n-1}} \end{array}\right],
\eeu
where $Z_\rho(n)$ is the same formal $*$-monomial matrix described above. When $Z_\rho(n)$ is viewed with entries in $C(\S^{2n}_\rho)$, it is normal and satisfies $Z_\rho(n) Z_\rho(n)^* = (\sum z_j z_j^*)I_{2^{n-1}} = (1 - x^2)I_{2^{n-1}}$, and it also commutes with $xI_{2^{n-1}}$. The projection $P_\rho(n)$ on the even sphere, much like its analogous unitary $Z_\rho(n)$ on the odd sphere, is set apart by its interaction with the antipodal action; the form
\beu
P_\rho(n) = \frac{1}{2}I_{2^n} + \frac{1}{2}\left[\begin{array}{cc} xI_{2^{n-1}} & Z_\rho(n) \\ Z_\rho(n)^* & -xI_{2^{n-1}} \end{array}\right]
\eeu
shows that $\alpha(P_\rho(n)) = I_{2^n} - P_\rho(n)$, where the antipodal action $\alpha$ is applied entrywise. Matrices of particular dimensions satisfying this identity will be distinguished in $K_0$. There is a marked advantage of odd $K$-theory for these types of problems, as the group operation of $K_1$ is realized as both direct sum and matrix multiplication, which allows results like Lemma \ref{lem:oddevenK_1} to manipulate the fixed point subalgebra instead of the $(-1)$-eigenspace. On the other hand, the group operation of $K_0$ is only realized as the direct sum, so we glean $K_0$ information on $C(\S^{2n}_\rho)$ from $K_1$ information of $C(\S^{2n + 1}_\omega)$.

If $P \in M_k(C(\S^{2n}_\omega))$ is a projection, then $B = 2P - I$ is a self-adjoint square root of $I$, or equivalently $B$ is self-adjoint and unitary. Through (\ref{eq:even=oddsuspension}), viewing $B$ as a function of $x \in [-1, 1]$ with (self-adjoint and unitary) values in $M_k(C(\S^{2n - 1}_\omega))$ allows us to form a unitary function $\widetilde{B}: \overline{\D} \to M_k(C(\S^{2n - 1}_\omega))$ by introducing an additional coordinate.
\beu
\widetilde{B}(x + iy) = \sqrt{1 - y^2}B\left(\frac{x}{\sqrt{1 - y^2}}\right) + iyI_k
\eeu
The function $\widetilde{B}$ appears to be only well-defined on $\overline{\D} \setminus \{\pm i\}$, but the  term $B\left(\cfrac{x}{\sqrt{1 - y^2}}\right)$ is bounded, as $B$ has norm $1$, and the multiplication by $\sqrt{1 - y^2}$ implies that $\widetilde{B}$ extends continuously to $\overline{\D}$ by the squeeze theorem. Because $B$ is self-adjoint and unitary, $\widetilde{B}$ is unitary. Moreover, the boundary conditions on $P$ and $B = 2P - I$ imply that if $x^2 + y^2 = 1$, then $\widetilde{B}(x + iy) \in M_k(\C)$, so $\widetilde{B}$ represents a unitary element of $M_k(C(\S^{2n + 1}_\rho))$ by (\ref{eq:odd=diskfunctions}). The association $P \mapsto B \mapsto \widetilde{B}$ between projections of dimension $k$ over $C(\S^{2n}_\omega)$ to unitaries of dimension $k$ over $C(\S^{2n + 1}_\rho)$ is continuous, meaning it respects path components, and in addition it is compatible with the direct sum. Also, if $P$ is a trivial projection $I_k \oplus 0_l$, then the function $\widetilde{B}$ takes values in $U_k(\C)$. The domain $\overline{\D}$ is contractible, so $\widetilde{B}$ is connected within $U_{k+l}(C(\S^{2n + 1}_\rho))$ to a scalar-entried matrix, which is then connected to the identity. (Note that this contraction was possible because the range of $\widetilde{B}$ only included scalar-valued matrices, which did not conflict with boundary conditions. In general an element in $U_{k + l}(C(\S^{2n + 1}_\rho))$ is not connected to the identity, even though $\overline{\mathbb{D}}$ is contractible.) In other words, if $\Omega_k$ denotes the function $P \mapsto \widetilde{B}$ for projections $P$ of dimension $k$, then the various $\Omega_k$ are compatible and produce a single homomorphism
\beu
\Omega: K_0(C(\S^{2n}_\omega)) \to K_1(C(\S^{2n + 1}_\rho))
\eeu
on the $K$-groups, and $\Omega$ has all trivial projections $I_k \oplus 0_l$ in its kernel. Moreover, each $\Omega_k$ is compatible with the $\Z_2$ antipodal action in the following way. If $P \in M_k(C(\S^{2n}_\omega))$ is a projection satisfying $\alpha_{2n}(P) = I - P$, then $B = 2P - I$ has $\alpha_{2n}(B) = -B$, or rather, $B$ is odd. This is reflected in the function algebra as $\alpha_{2n}(B)[x] = \alpha_{2n - 1}(B(-x)) = -B(x)$, which will help show that $\Omega_k(P) = \widetilde{B}$ is also odd.
\beu \ba
\alpha_{2n + 1}(\widetilde{B})[x + i y]	 	&= \alpha_{2n - 1}(\widetilde{B}(-x - iy)) \\
							&= \alpha_{2n - 1}\left(\sqrt{1 - (-y)^2}B\left(\frac{-x}{\sqrt{1 - (-y)^2}}\right) + i(-y)I_k\right) \\
							&= \sqrt{1 - y^2} \left[\alpha_{2n - 1}\left(B\left(\frac{-x}{\sqrt{1 - y^2}}\right)\right)\right] - iyI_k \\
							&= \sqrt{1 - y^2} \left(-B\left(\frac{x}{\sqrt{1 - y^2}}\right)\right) - iyI_k \\
							&=-\left( \sqrt{1 - y^2}B\left(\frac{x}{\sqrt{1 - y^2}}\right) + iyI_k \right) \\
							&= -\widetilde{B}(x + iy)
\ea \eeu

\begin{thm}\label{thm:K0oddprojections}
If $P$ is a projection matrix over $C(\S^{2n}_\omega)$ of dimension $(2m + 1)2^{n}$ that satisfies $\alpha(P) = I - P$, then the class of $P$ in $K_0(C(\S^{2n}_\rho)) \cong \Z \oplus \Z$ is not in the subgroup generated by trivial projections.
\end{thm}
\begin{proof}
Suppose $P$ is in the subgroup generated by trivial projections, and consider $V := \Omega_{(2m + 1)2^{n}}(P)$ in $U_{(2m + 1)2^{n}}(C(\S^{2n + 1}_\rho))$, where $\rho$ is an $(n + 1) \times (n + 1)$ parameter matrix with $\omega$ in the upper left and $1$ in all other entries. It follows that $[V]_{K_1} = \Omega([P]_{K_0})$ is the trivial element of $K_1(C(\S^{2n + 1}_\rho))$. However, $V$ is odd and of dimension $(2m + 1)2^n$, so this contradicts Lemma \ref{lem:oddevenK_1} for spheres of dimension $2n + 1 = 2(n + 1) - 1$.
\end{proof}

This theorem shows that the $\Z_2$ structure of even spheres is reflected in their $K$-theory. Just as in the odd case, we now prove a noncommutative Borsuk-Ulam theorem.

\begin{cor}\label{cor:eventhetaBU}
Suppose $\Phi: C(\S^{2n}_\rho) \to C(\S^{2n}_\omega)$ is a unital $*$-homomorphism that is equivariant for the antipodal maps. If $K_0 \cong \Z \oplus \Z$ is such that the first summand is generated by trivial projections, then the $K_0$ map induced by $\Phi$ is nontrivial on the second summand.
\end{cor}
\begin{proof}
Consider the projection $P_\rho(n) = \left[\begin{array}{cc} \frac{1 + x}{2}I_{2^{n - 1}} & Z_\rho(n) \\ Z_\rho(n)^* & \frac{1 - x}{2}I_{2^{n - 1}} \end{array}\right]$, which is of dimension $2^n$ and satisfies $\alpha(P_\rho(n)) = I - P_\rho(n)$. Because $\Phi$ is equivariant, the same applies to $\Phi(P_\rho(n))$, so $\Phi(P_\rho(n))$ is not in the subgroup generated by trivial projections.
\end{proof}

\section{Anticommutation}

The same argument to establish (\ref{eq:even=oddsuspension}), also seen in \cite{pa15b}, can show that $\fR^{2n}_\omega$ is a noncommutative unreduced suspension of $C(\S^{2n-1}_\omega)$.
\be\label{eq:antipaths}
\fR^{2n}_\omega \cong \Sigma^\alpha C(\S^{2n-1}_\omega) := \{f \in C([0, 1], C(\S^{2n-1}_\omega) \rtimes_\alpha \Z_2): f(0) \in C(\S^{2n-1}_\omega), f(1) \in C^*(\Z_2)\}
\ee
In the above, $\alpha: C(\S^{2n-1}_\omega) \to C(\S^{2n-1}_\omega)$ denotes the antipodal map, the unique homomorphism with $\alpha(z_j) = -z_j$. The generators $z_1, \ldots, z_n, x$ in the presentation (\ref{eq:antipres}) of $\fR^{2n}_\omega$ then take the following function forms.
\beu
x \sim X(t) = t \delta
\eeu 
\beu
z_j \sim Z_j(t) = \sqrt{1-t^2}z_j
\eeu
The computation of even $K$-theory for $\fR^{2n}_\omega$ will be aided by the following matrix expansion map.

\begin{defn}
The \textit{expansion map} $E_\alpha: C(\S^{2n - 1}_\omega) \rtimes_\alpha \Z_2 \to M_2(C(\S^{2n - 1}_\omega))$ is the injective, unital $*$-homomorphism defined by the following rule.
\beu
f + g \delta \mapsto \left[\begin{array}{cc} f & g \\ \alpha(g) & \alpha(f) \end{array} \right]
\eeu
\end{defn}

See \cite{wi07}, section 2.5 for a similar map. The map $E_\alpha$ is certainly not surjective, as all matrices in its range satisfy a very visible symmetry condition.

\begin{prop}\label{prop:expconj}
If $M = \left[\begin{array}{cc} f & g \\ \alpha(g) & \alpha(f) \end{array} \right] \in \textrm{Ran}(E_\alpha)$, then $M$ satisfies the symmetry condition $M = \left[\begin{array}{cc} 0 & 1 \\ 1 & 0 \end{array} \right]^* \alpha(M)  \left[\begin{array}{cc} 0 & 1 \\ 1 & 0 \end{array} \right]$. The $K_1(C(\S^{2n-1}_\omega))$ class of $M$ is then represented by an even integer.
\end{prop}
\begin{proof}
The symmetry condition is apparent, and we may apply \cite{pa15a}, Theorem 3.10.
\end{proof}

The antipodal action $\alpha$ is saturated, so as a result of Morita equivalence, $K$-theory of the crossed product $C(\S^{2n - 1}_\omega) \rtimes_\alpha \Z_2$ is isomorphic to the $K$-theory of the fixed point subalgebra $C(\S^{2n - 1}_\omega)^\alpha$, the algebra of even elements. This is useful when considering the ideal $J$ of $\fR^{2n}_\omega$ consisting of functions which vanish at the endpoints $0$ and $1$, so $J \cong S(C(\S^{2n - 1}_\omega) \rtimes_\alpha \Z_2)$ and consequently $K_i(J) \cong K_i(S(C(\S^{2n - 1}_\omega) \rtimes_\alpha \Z_2))) \cong K_{1 - i}(C(\S^{2n - 1}_\omega) \rtimes_\alpha \Z_2) \cong K_{1 - i}(C(\S^{2n - 1}_\omega)^\alpha)$. The fixed point subalgebra $C(\S^{2n - 1}_\omega)^\alpha$ is a Rieffel deformation of $C(\mathbb{RP}^{2n - 1}) \cong C(\S^{2n-1})^\alpha$ and therefore has identical $K$-theory ($K$-theory of real projective space is computed in \cite{at67}, Proposition 2.7.7). Boundary information is contained in the quotient $\fR^{2n}_\omega / J \cong C(\S^{2n - 1}_\omega) \oplus (\C \oplus \C \delta)$, and $K$-theory respects direct sums, so the six term exact sequence yields the following.
\be\label{eq:Kanticomm}
 \begin{tikzpicture}
  \matrix (m) [matrix of math nodes,row sep=3em,column sep=4em,minimum width=2em]
  {
     K_0(S(C(\S^{2n - 1}_\omega) \rtimes_\alpha \Z_2)) & K_0(\fR^{2n}_\omega) & K_0(C(\S^{2n - 1}_\omega)) \oplus K_0(\C + \C\delta) \\
     K_1(C(\S^{2n - 1}_\omega)) \oplus K_1(\C + \C \delta) & K_1(\fR^{2n}_\omega) &   K_1(S(C(\S^{2n - 1}_\omega) \rtimes_\alpha \Z_2))\\};
  \path[-stealth]
    (m-2-1) edge node [left] {$\eta$} (m-1-1)
    (m-1-1) edge node [above] {$\gamma$} (m-1-2)
    (m-2-2) edge node [above] {$\lambda$} (m-2-1)
    (m-1-2) edge node [above] {$\upsilon$} (m-1-3)
    (m-2-3) edge node [above] {$\kappa$} (m-2-2)
    (m-1-3) edge node [right] {$\sigma$} (m-2-3);
\end{tikzpicture}
\ee

We repeat this sequence below with the known isomorphism classes and use it to calculate $K_0(\fR^{2n}_\omega)$.
\be\label{eq:Kanticomm2}
 \begin{tikzpicture}
  \matrix (m) [matrix of math nodes,row sep=3em,column sep=4em,minimum width=2em]
  {
    \Z & K_0(\fR^{2n}_\omega) & \Z \oplus \Z \oplus \Z \\
   \Z & K_1(\fR^{2n}_\omega) &  \Z \oplus \Z_{2^{n - 1}} \\};
  \path[-stealth]
    (m-2-1) edge node [left] {$\eta$} (m-1-1)
    (m-1-1) edge node [above] {$\gamma$} (m-1-2)
    (m-2-2) edge node [above] {$\lambda$} (m-2-1)
    (m-1-2) edge node [above] {$\upsilon$} (m-1-3)
    (m-2-3) edge node [above] {$\kappa$} (m-2-2)
    (m-1-3) edge node [right] {$\sigma$} (m-2-3);
\end{tikzpicture}
\ee

\begin{prop}
The group $K_0(\fR^{2n}_\omega)$ is isomorphic to $\Z \oplus \Z$, and if matrix projections $P \in M_k(\fR^{2n}_\omega)$ are considered as paths from (\ref{eq:antipaths}), the $K_0$ data is obtained from the two ranks $\textrm{Rank}_{\delta = \pm 1}(P(1))$.
\end{prop}
\begin{proof}
The subalgebra $C^*(\Z_2) = \C + \C \delta \leq C(\S^{2n - 1}_\omega) \rtimes_\alpha \Z_2$ is two dimensional and isomorphic to a continuous function algebra on two points, $C(\{ \pm 1\})$, so an evaluation of $\delta$ at $1$ or $-1$ can be used on this subalgebra (however, an evaluation does not make sense on the entire crossed product). Moreover, $K_0(C(\S^{2n - 1}_\omega) \oplus (\C + \C\delta)) \cong \Z \oplus \Z \oplus \Z$ is essentially rank data in three instances: a rank for $K_0(C(\S^{2n - 1}_\omega)) \cong \Z$ and two ranks for $K_0(C(\{\pm1\})) \cong \Z \oplus \Z$ based on evaluations $\delta = \pm 1$. However, the map $\upsilon$ in (\ref{eq:Kanticomm2}) is not surjective by the following argument. First, the projections $\cfrac{1 \pm \delta}{2} \in \C + \C \delta$ that generate its $K_0$ group are sent under $E_\alpha$ to rank one projections $\left[\begin{array}{cc}1/2 & \pm 1/2 \\ \pm 1/2 & 1/2 \end{array}\right]$ in $M_2(C(\S^{2n - 1}_\omega))$, so in $K_0(M_2(C(\S^{2n - 1}_\omega))) \cong K_0(C(\S^{2n - 1}_\omega)) \cong \Z$ these are the generator $1$. Next, if we consider $C(\S^{2n - 1}_\omega)$, whose $K_0$ group is again cyclic and generated by the trivial projection $1$, the image $E_\alpha(1) = I_2$ has doubled in rank, and the same will happen to any projection over $C(\S^{2n - 1}_\omega)$ when $E_\alpha$ is applied. Now, any projection matrix over $\fR^{2n}_\omega$ is a path of projection matrices over $C(\S^{2n - 1}_\omega) \rtimes_\alpha \Z_2$, so the $K_0$ class of the images remains constant along the path. Based on the previous computations, if $P \in M_k(\fR^{2n}_\omega)$ is viewed as a path into $M_k(C(\S^{2n - 1}_\omega) \rtimes_\alpha \Z_2)$, then there is the following restriction on the ranks of $P(t)$.
\beu \ba
2 \cdot \textrm{Rank}_{C(\S^{2n - 1}_\omega)}(P(0)) 	&= \textrm{Rank}_{C(\S^{2n - 1}_\omega)}(E_\alpha(P(0))) \\
									&= \textrm{Rank}_{C(\S^{2n - 1}_\omega)}(E_\alpha(P(1))) \\
									&= \textrm{Rank}_{\delta = 1}(P(1)) + \textrm{Rank}_{\delta = -1}(P(1))
\ea \eeu

From this equation, we reach constraints on the range of $\upsilon$ in (\ref{eq:Kanticomm2}).
\be
(l, m, n) \in \textrm{Ran}(\upsilon) \implies 2l = m + n
\ee
The first clear projection in $\textrm{Ran}(\upsilon)$ is the identity element, which produces the predictable ranks $(1, 1, 1)$. The second is based on adjusting coefficients from a projection over $\theta$-deformed even spheres; in $\fR^{2n}_\omega$, $x$ anticommutes with each $*$-monomial entry of $Z_\omega(n)$.
\be\label{eq:what}
P := P^\prime_\omega(n) := \cfrac{1}{2}\hspace{2 pt}I_{2^n} + \cfrac{1}{2} \left[\begin{array}{cc} xI_{2^{n - 1}} & Z_\omega(n) \\ Z_\omega(n)^* & xI_{2^{n - 1}} \end{array} \right]
\ee
When viewed as a path, the projection $P$ has $P(0) = \cfrac{1}{2} \left[\begin{array}{cc} I_{2^{n - 1}} & Z_\omega(n) \\ Z_\omega(n)^* & I_{2^{n - 1}} \end{array}\right]$, which has rank $2^{n - 1}$ in $K_0(C(\S^{2n - 1}_\omega))$, and $P(1) = \frac{1 + \delta}{2} I_{2^n}$, which has rank $2^n$ at $\delta = 1$ and rank $0$ at $\delta = -1$. Therefore, $P$ produces the element $(2^{n-1}, 2^{n}, 0) \in \textrm{Ran}(\upsilon)$. The tuples $(1, 1, 1)$ and $(2^{n-1}, 2^{n}, 0)$ are independent, so their span has free rank $2$, but for large $n$ the second tuple is not in reduced form, meaning we cannot tell if $2l = m + n$ completely characterizes elements of $\textrm{Ran}(\upsilon)$ (that is, if $\textrm{Ran}(\upsilon)$ is a full or proper subset of $\textrm{span}_\Z \{(1, 1, 1), (1, 2, 0)\}$).
\beu
\textrm{span}_\Z \{(1, 1, 1), 2^{n-1}(1, 2, 0)\} \leq \textrm{Ran}(\upsilon) \leq \textrm{span}_\Z \{(1, 1, 1), (1, 2, 0) \}
\eeu

Next, note that $K_1(C(\S^{2n - 1}_\omega) \rtimes_\alpha \Z_2) \cong K_1(C(\RP^{2n - 1})) \cong \Z$ is generated by $Z_\omega(n)$, which is again seen through $E_\alpha$. The expansion $E_\alpha(Z_\omega(n))$ is of index $2$ because it is unitarily equivalent to $Z_\omega(n) \oplus Z_\omega(n)$, and by Proposition \ref{prop:expconj}, any invertible matrix in $\textrm{Ran}(E_\alpha)$ will correspond to an even integer in $K_1(C(\S^{2n - 1}_\omega))$, so an index 2 matrix in $K_1(\textrm{Ran}(E_\alpha)) \cong K_1(C(\S^{2n - 1}_\omega) \rtimes_\alpha \Z_2) \cong \Z$ is a generator. Since $E_\alpha$ is an isomorphism between $C(\S^{2n - 1}_\omega \rtimes_\alpha \Z_2)$ and $\textrm{Ran}(E_\alpha)$, this implies that $K_1(C(\S^{2n - 1}_\omega) \rtimes_\alpha \Z_2)$ is generated by $Z_\omega(n)$. The connecting map $\eta$ between  $K_1(C(\S^{2n - 1}_\omega) \oplus (\C + \C \delta)) \cong K_1(C(\S^{2n - 1}_\omega))$ and $K_0(S(C(\S^{2n - 1}_\omega) \rtimes_\alpha \Z_2))$ from the six term sequence mimics the form of the isomorphism $K_1(C(\S^{2n - 1}_\omega) \rtimes_\alpha \Z_2) \to K_0(S(C(\S^{2n - 1}_\omega \rtimes_\alpha \Z_2)))$ of Bott periodicity, but now that we know $K_1(C(\S^{2n - 1}_\omega) \rtimes_\alpha \Z_2)$ is generated by a unitary over $C(\S^{2n - 1}_\omega)$ alone, we can conclude that $\eta$ is surjective. Therefore $\gamma$ is the zero map and $\upsilon$ is injective, so $K_0(\fR^{2n}_\omega)$ is isomorphic to $\textrm{Ran}(\upsilon) \cong \Z \oplus \Z$.
\end{proof}

Even though the $K_0$ groups of $\fR^{2n}_\omega$ and $C(\S^{2n}_\omega)$ are abstractly isomorphic, the generators are different in that $K_0$ data of $\fR^{2n}_\omega$ is contained in the possibly distinct ranks $\textrm{Rank}_{\delta = 1}$ and $\textrm{Rank}_{\delta = -1}$, whereas $K_0(C(\S^{2n}_\omega))$ inherits its $K$-theory from $K_0(C(\S^{2n}))$, which has one summand for rank and another for a nontrivial vector bundle. Despite all of this, the form of the nontrivial projection $P^\prime_\omega(n)$ in $(\ref{eq:what})$ is just a slight sign change from a projection $P_\omega(n)$ over $C(\S^{2n}_\omega)$ seen in the proof of Corollary \ref{cor:eventhetaBU}. The distinguished projection
\be\label{eq:distinguishedprojection}
P^\prime_\omega(n) := \cfrac{1}{2}\hspace{2 pt}I_{2^n} + \cfrac{1}{2} \left[\begin{array}{cc} xI_{2^{n - 1}} & Z_\omega(n) \\ Z_\omega(n)^* & xI_{2^{n - 1}} \end{array} \right]
\ee
is also compatible with the antipodal map, in perfect analogy with the projection $P_\omega(n)$ on $C(\S^{2n}_\omega)$. Note that the antipodal map on $\fR^{2n}_\omega \cong \Sigma^\alpha C(\S^{2n-1}_\omega)$ is given in path form as
\beu
\mathfrak{a}(f)[t] = \alpha \widehat{\alpha}(f(t)) = \widehat{\alpha}\alpha(f(t)),
\eeu
where $\widehat{\alpha}: a + b \delta \mapsto a - b \delta$ is the dual action on $C(\S^{2n-1}_\omega) \rtimes_\alpha \Z_2$. The ranks $\textrm{Rank}_{\delta = \pm 1}$ at $t = 1$ take different values for $P^\prime_\omega(n)$, as evaluation at $t = 1$ yields $\frac{1 + \delta}{2} I_{2^{n}}$. Based on this example, we can ask if any $2^n \times 2^n$ projection in $\fR^{2n}_\omega$ satisfying $\mathfrak{a}(P) = I - P$ is nontrivial in $K_0$. This result does hold, in analogy with the $\theta$-deformed even spheres, but the proof is necessarily quite different, as the suspension argument of the previous section does not pass over cleanly to this case. In order to rectify this problem, we manipulate a particular non-free action on commutative spheres.

\begin{thm}\label{thm:gammaustar}
 Let the commutative sphere $C(\S^k)$ be equipped with a $\Z_2$ action $\gamma$ that fixes a distinguished real-valued coordinate $h$ but negates the remaining coordinates $x_1, \ldots, x_{k}$, and let $U$ be a unitary matrix over $C(\S^k)$ with $\gamma(U) = U^*$ and $U|_{h = \pm 1} = I$. Then the following hold.
\begin{enumerate}
\item If $U$ represents a trivial element in $K_1(C(\S^k))$, then there is a path connecting some stabilization $U \oplus I$ to $I$ within the set of unitaries that satisfy $\gamma(M) = M^*$ and $M|_{h = \pm 1} = I$.

\item If $k$ is odd, then $U$ corresponds to an even integer in $K_1(C(\S^k)) \cong \Z$.
\end{enumerate}
\end{thm}
\begin{proof}
We proceed by induction; the only applicable unitaries over $C(\S^0)$ are the identity matrices, so the claim at dimension $0$ is automatically satisfied. If the claim holds for $k$, then suppose $U$ is a unitary matrix over $C(\S^{k + 1})$ with $\gamma(U) = U^*$ and $U|_{h = \pm 1} = I$. Let $\S^k$ denote the equator in $\S^{k+1}$ determined by $x_k = 0$ (noting $x_k$ is a coordinate negated by $\gamma$), so that the restriction $\widetilde{\gamma}$ of $\gamma$ to $C(\S^k)$ is an action of the same type: it fixes $h$ and negates $x_1, \ldots, x_{k-1}$. Also note that the fixed points $h = \pm 1$ are in this equator $\S^k$. To use the inductive hypothesis, first form a path of unitaries that \lq\lq stretches\rq\rq \hspace{0pt} the equator data of $U$. Realize $C(\S^{k + 1})$ as the unreduced suspension $\Sigma C(\S^k)$, with $x_k$ representing the path coordinate in $[-1, 1]$. Then $U = U(x_k)$ represents a path of unitary matrices over $C(\S^k)$, and we form a continuous path $U_t$ as follows.

\beu
U_t(x_k) = \left\{ \begin{array}{ccc} U\left(\frac{-2}{t-2}(x_k + 1) - 1 \right) & \textrm{if} & -1 \leq x_k \leq \frac{-t}{2} \\ U(0) & \textrm{if} & \frac{-t}{2} \leq x_k \leq \frac{t}{2} \\ U\left(\frac{-2}{t-2}(x_k - 1) + 1 \right) & \textrm{if} & \frac{t}{2} \leq x_k \leq 1 \end{array}\right.
\eeu

This path connects $U = U_0$ to $V = U_1$, which has its equator data repeated on a band neighborhood $-\frac{1}{2} \leq x_k \leq \frac{1}{2}$. The path $U_t$ also satisfies $\gamma(U_t) = U_t^*$ and still assigns the identity on $h = \pm 1$, as these points are in the equator $x_k = 0$. The equator function $U(0) = V(0)$ is trivial in $K_1(C(\S^k))$ because $\S^k$ sits inside a contractible subset of $\S^{k + 1}$; it also satisfies $\widetilde{\gamma}(U(0)) = U(0)^*$ and sends the fixed points at $h = \pm 1$ to the identity matrix. By the inductive hypothesis, there is a path $W_t$ connecting $W_0 = U(0) \oplus I$ to $W_1 = I$ within the unitaries over $C(\S^k)$ that also satisfy $\widetilde{\gamma}(W_t) = W_t^*$ and assign the fixed points $h = \pm 1$ to the identity matrix. Apply this path to the equator of $V \oplus I$ while maintaining the same values on the path for $|x_k| \geq \frac{1}{2}$.
\beu
V_t(x_k) = \left\{ \begin{array}{ccc} W_{\phi_t(x_k)} & \textrm{if} & |x_k| \leq \frac{1}{2} \\ V(x_k) \oplus I & \textrm{if} & \frac{1}{2} \leq |x_k| \leq 1 \end{array}\right.
\eeu
\beu
\phi_t(x_k) = \left\{ \begin{array}{cccc} -2|x_k| + t & \textrm{if} & |x_k| \leq \frac{t}{2} \\ 0 & \textrm{if} & \frac{t}{2} \leq |x_k| \leq \frac{1}{2} \end{array}\right.
\eeu

The path $V_t$ connects $V_0 = V \oplus I$ to $V_1$, where the unitaries $V_t$ still satisfy $\gamma(V_t) = V_t^*$, and by the inductive assumption, the fixed points $h = \pm 1$ of $\gamma$ are still always assigned the identity. Because $V_1$ also assigns the identity matrix on the entire equator, $V_1$ is the commuting product of two unitary matrices $F$ and $G$, where $F$ assigns the identity matrix for $x_k \geq 0$, $G$ assigns the identity matrix for $x_k \leq 0$, and $\gamma(F^*) = G$. How to proceeed is slightly different based on the parity of $k + 1$.

If $k + 1$ is even, then $K_1(C(\S^{k + 1}))$ is the trivial group, so let $F_t$ denote a path of unitaries connecting $F_0 = F \oplus I$ to $I$. Moreover, we can insist that $F_t$ always assigns the identity matrix on points with $x_k \geq 0$, as this region is contractible to a point and $K_1(C(\S^{k+1} \setminus \{\textrm{pt}\}))$ is also trivial. Then the commuting product $Q_t = F_t \cdot \gamma(F_t^*)$ forms a path of unitaries that connects $V_1 \oplus I$ to $I$, satisfies $\gamma(Q_t) = Q_t^*$, and assigns the identity matrix at (at least) $h = \pm 1$. Tracing back all of the paths used so far shows that $U \oplus I$ is connected to $I$ within the same set of restricted unitaries.

If $k + 1$ is odd, then $K_1(C(\S^{k + 1})) \cong \Z$, and as $\gamma$ and the adjoint are both orientation-reversing, $F$ and $G = \gamma(F^*)$ represent the same element in $K_1(C(\S^{k + 1}))$. The product $V_1 = F \cdot \gamma(F^*)$ shows that the class of $V_1$ (and therefore of $U$) in $K_1$ is $2[F]_{K_1}$, an even integer. If this integer is zero, then once again, $K_1(C(\S^{k + 1} \setminus \{\textrm{pt}\}))$ is isomorphic to $K_1(C(\S^{k + 1}))$, so just as in the previous paragraph, the trivial element $F \oplus I$ can be connected to $I$ in a path of unitaries $F_t$ that always assign the identity on points with $x_k \geq 0$. Finally, the commuting product $Q_t = F_t \cdot \gamma(F_t^*)$ connects $V_1 \oplus I$ to $I$, satisfies $\gamma(Q_t) = Q_t^*$, and always send the points $h = \pm 1$ to the identity matrix. A composition of paths establishes the same for $U$ and completes the induction.
\end{proof}

In the above theorem, the set of matrices $M$ satisfying $\gamma(M) = M^*$ is not a $C^*$-subalgebra of the matrix algebra over $C(\S^{k})$, as the adjoint operation reverses the order of multiplication. As such, some of the ideas in the proof are motivated by the six term exact sequence, but cannot be implemented this way. Mention of function values at the fixed points $h = \pm 1$ appears to be unnecessary on first glance, but this is crucial even at low dimensions.

\begin{exam}
Let $C(\S^1)$ be generated by the complex coordinate $z = x + iy$ with a $\Z_2$ action $\gamma$ that negates $y$ but fixes $x$. Then $\gamma(z) = z^*$, but $z$ is the generator of $K_1(C(\S^1))$, associated to the odd integer $1$. The fixed points are $x = 1$ and $x = -1$, and $z$ assigns these points to $\pm 1$. The previous theorem does not apply, as it only considers matrices which assign the identity matrix on the fixed points of $\gamma$. The same scenario happens in any $C(\S^{2n - 1})$ for the standard $K_1$ generator $Z(n)$.
\end{exam}

Theorem \ref{thm:gammaustar} concerns an action $\gamma$ that is not free, but it will be instrumental in defining an invariant for unitaries that satisfy $\alpha(U) = U^*$, where $\alpha$ is the antipodal map on $C(\S^{2n - 1})$. In turn, this invariant will show that the relation algebras $\fR^{2n}_\omega$ have a Borsuk-Ulam property for projections in $K_0$, by first proving the case when $z_1, \ldots, z_n$ all commute with each other. 

\begin{defn}
Fix an isomorphism $C(\S^{k}) \cong \Sigma C(\S^{k-1})$, with the path coordinate denoted by $h$. If $M \in M_p(C(\S^k))$ is such that $M(0) \in M_p(\C)$ (i.e., it has constant entries), then let $M^+ \in M_p(C(\S^k)) \cong M_p(\Sigma C(\S^{k-1}))$ be defined as follows for $h \in [-1, 1]$.
\beu
M^+(h) = M\left(\frac{h + 1}{2}\right)
\eeu
The matrix $M^+$ encodes information about $M$ only for points $h \geq 0$. Note in particular that $M$ is also only defined in reference to a particular, fixed coordinate $h$ and a suppressed isomorphism with an unreduced suspension. We do not expect that any $M^+$ that is definable with respect to multiple coordinate choices has any uniqueness properties of any kind.
\end{defn}

\begin{thm}\label{thm:unexpectedZ2}
Suppose $U_0, U_1 \in U_k(C(\S^{2n - 1}))$ are unitary matrices which assign an equator $h = 0$ to the identity matrix and satisfy $\alpha(U_j) = U_j^*$. If there is a path of unitary matrices $U_t$ connecting $U_0$ and $U_1$ that satisfy $\alpha(U_t) = U_t^*$ (but have no restriction on the equator), then the classes of $U_0^+$ and $U_1^+$ in $K_1(C(\S^{2n - 1})) \cong \Z$ differ by an even integer.
\end{thm}
\begin{proof}
Fix an isomorphism $C(\S^{2n - 1}) \cong \Sigma C(\S^{2n - 2})$, with the path coordinate specified by $h$, and let $\gamma$ be a $\Z_2$ action on $C(\S^{2n - 1})$ that fixes $h$ but applies the antipodal map pointwise on $C(\S^{2n - 2})$. Define $V_t \in U_k(C(\S^{2n - 1}))$ as follows.
\beu
V_t(h) = \left\{\begin{array}{cccc} U_{-h}(0) & \textrm{if} & -1 \leq h \leq -t \\ U_t(\frac{h - 1}{1 + t} + 1) & \textrm{if} & -t \leq h \leq 1 \end{array} \right.
\eeu

Note that $V_t(-1) = U_1(0)$ is always the identity matrix and $V_t(1) = U_t(1)$ is a constant matrix, meeting the necessary boundary conditions to define a unitary over the sphere. Further, $V_1$ is equal to $U_1^+$, and $V_0$ contains data of $U_0^+$ with equator information of $U_t$.
\beu
V_0(h) = \left\{\begin{array}{cccc} U_{-h}(0) & \textrm{if} & -1 \leq h \leq 0 \\ U_0(h) & \textrm{if} & 0 \leq h \leq 1 \end{array} \right.
\eeu
Since $U_0(0)$ is the identity matrix, this formula shows that $V_0$ can be written as a commuting product of two unitaries $F \cdot G$, as follows.
\beu
F(h) = \left\{\begin{array}{cccc} I & \textrm{if} & -1 \leq h \leq 0 \\ U_0(h) & \textrm{if} & 0 \leq h \leq 1 \end{array} \right. \hspace{1 cm} \implies \hspace{1 cm} [F]_{K_1} = [U_0^+]_{K_1}
\eeu
\beu
G(h) = \left\{\begin{array}{cccc} U_{-h}(0) & \textrm{if} & -1 \leq h \leq 0 \\ I & \textrm{if} & 0 \leq h \leq 1 \end{array} \right.
\eeu

Now, the path matrices $U_t$ satisfy $\alpha(U_t) = U_t^*$, and the matrix $G$ contains equator data from $U_t$, so it satisfies $\gamma(G) = G^*$ where $\gamma$ negates every coordinate in $C(\S^{2n - 1})$ except $h$. That is, $\gamma$ applies the antipodal map of $C(\S^{2n - 2})$ pointwise on $\Sigma C(\S^{2n - 2})$. $G$ also assigns the fixed points of $\gamma$, $h = \pm 1$, to the identity matrix, so the class of $G$ in $K_1(C(\S^{2n - 1}))$ is an even integer by Theorem \ref{thm:gammaustar}. The formulas $V_0 = F \cdot G$ and $[F]_{K_1} = [U_0^+]_{K_1}$ imply that $[V_0]_{K_1}$ is an even integer away from $[U_0^+]_{K_1}$, and $[V_0]_{K_1}$ is the same as $[V_1]_{K_1} = [U_1^+]_{K_1}$.
\end{proof}

The above theorem defines a $\Z_2$ invariant for any unitary over $C(\S^{2n - 1})$ satisfying $\alpha(U) = U^*$ which assigns the identity on a specified equator $h = 0$: $[U^+]_{K_1} \textrm{ mod } 2$. This invariant is preserved by paths $U_t$ satisfying $\alpha(U_t) = U_t^*$ where $U_0$ and $U_1$ assign the identity on the equator. It is important to note that only the endpoints of the path have specified equator data; if the equators of $U_t$ were also the identity, $U_t^+$ would be a well-defined path and the $K_1$ classes of $U_0^+$ and $U_1^+$ would be equal, with no need for reducing mod $2$. Some examples of unitary paths with $\alpha(U_t) = U_t^*$ will be necessary for later calculations.

\begin{exam}\label{exam:identitypaththing}
Fix an isomorphism $C(\S^{2n - 1}) \cong \Sigma C(\S^{2n - 2})$ with path coordinate $x_1 = \textrm{Re}(z_1)$, and consider the standard $K_1$ generator $Z(n) \in U_{2^{n - 1}}(C(\S^{2n - 1}))$. Define a unitary $V \in U_{2^{n - 1}}(C(\S^{2n - 1}))$ as follows.
\beu
V(x_1) = -Z(n)[2|x_1| - 1]
\eeu
Now, $V$ has $V(\pm 1) = -Z(n)[1] = -I$ and $V(0) = -Z(n)[-1] = I$, and the antipodal map $\alpha$ on $C(\S^{2n - 1})$ comes from applying $x_1 \mapsto -x_1$ and the pointwise antipodal map $\widetilde{\alpha}$ on $C(\S^{2n - 2})$.
\beu
\alpha(V)[x_1] = \widetilde{\alpha}(V(-x_1)) = \widetilde{\alpha}(V(x_1)) = \widetilde{\alpha}(-Z(n)[2|x_1| - 1])
\eeu
The application of $\widetilde{\alpha}$ to $-Z(n)[2|x_1| - 1]$ negates $y_1 = \textrm{Im}(z_1)$ and $z_2, \ldots, z_n$, which produces $(-Z(n)[2|x_1| - 1])^*$ by an inductive argument, so $\widetilde{\alpha}(V(x_1)) = V(x_1)^*$ for each $x_1$. Finally, the above equation implies that $\alpha(V) = V^*$. Since $V(0) = I$ and $V^+ = -Z(n)$, it follows that $[V^+]_{K_1} = 1$. Moreover, $V$ is connected to $-I_{2^{n - 1}}$ within the unitaries satisfying $\alpha(M) = M^*$ by considering the following path $V_t$.
\beu
V_t(x_1) = V((1 -t)|x_1| + t)
\eeu
Now, $V_0(x_1) = V(|x_1|) = V(x_1)$, so $V_0 = V$, and $V_1(x_1) = V(1) = -I_{2^{n - 1}}$.
\beu \ba
\alpha(V_t)[x_1] 	&= \widetilde{\alpha}(V_t(-x_1)) \\
			&= \widetilde{\alpha}(V_t(x_1)) \\
			&= \widetilde{\alpha}(V((1-t)|x_1| + t)) \\
			&= V((1-t)|x_1| + t)^* \\
			&= V_t(x_1)^* \\
\ea \eeu
Finally, $V$ is connected to $-I_{2^{n-1}}$ within the unitary matrices satisfying $\alpha(M) = M^*$.
\end{exam}

\begin{exam}\label{exam:antioddpaththing}
Fix an isomorphism $C(\S^{2n - 1}) \cong \Sigma C(\S^{2n - 2})$ with path coordinate $x_1$. Suppose $U_0 \in U_k(C(\S^{2n - 1}))$ is anti self-adjoint ($U_0^* = -U_0$) and odd $(\alpha(U_0) = -U_0)$, which implies that $\alpha(U_0) = U_0^*$. Define a path $U_t$ as follows.

\beu
U_t(x_1) = \left\{\begin{array}{cccc} \widetilde{\alpha}(U_0(\frac{2}{2 - t}(-x_1 - 1) + 1)^*) & \textrm{ if } & -1 \leq x_1 \leq -\frac{t}{2} \\  (t + 2x_1)I + \sqrt{1 - (t + 2x_1)^2}\hspace{2 pt}\widetilde{\alpha}(U_0(0)^*) & \textrm{if}& -\frac{t}{2} \leq x_1 \leq 0  \\ (t - 2x_1)I + \sqrt{1 - (t - 2x_1)^2} \hspace{2 pt} U_0(0) & \textrm{if} & 0 \leq x_1 \leq \frac{t}{2} \\ U_0(\frac{2}{2 - t}(x_1 - 1) + 1) & \textrm{if} & \frac{t}{2} \leq x_1 \leq 1 \end{array} \right.
\eeu

Note that $U_0(0)$ is anti self-adjoint and odd under the antipodal action $\widetilde{\alpha}$ on $C(\S^{2n - 2})$, so $U_t$ is well-defined and unitary with $\alpha(U_t) = U_t^*$. Let $U_1 = W_0$ and define another path of unitaries $W_t$ which have $\alpha(W_t^*) = W_t$ and also assign the equator to $I$.
\beu
W_t(x_1) = \left\{ \begin{array}{cccc} \widetilde{\alpha}(U_0(-2x_1 - 1)^*) & \textrm{if} & -1 \leq x_1 \leq -\frac{1 + t}{2}  \\ (1 + \frac{2}{1 + t}x_1) I + \sqrt{1 - (1 + \frac{2}{1 + t}x_1)^2} \hspace{2 pt} \widetilde{\alpha}(U_0(t)^*) & \textrm{if} & -\frac{1 + t}{2} \leq x_1 \leq 0 \\ (1 -\frac{2}{1 + t}x_1) I + \sqrt{1 - (1 -\frac{2}{1 + t}x_1)^2} \hspace{2 pt} U_0(t) & \textrm{if} & 0 \leq x_1 \leq \frac{1 + t}{2} \\ U_0(2x_1 - 1) & \textrm{if} & \frac{1 + t}{2} \leq x_1 \leq 1 \end{array} \right.
\eeu
Note that the formula defining $W_t$ produces a unitary matrix because each $U_0(t)$ is anti self-adjoint, as is each $\widetilde{\alpha}(U_0(t)^*)$. Finally, $W_1(0) = I$, and because $U_0(1)$ has scalar entries, $W_1(x_1)$ has scalar entries for each $x_1$, which implies that $[W_1^+]_{K_0} = 0$.

\end{exam}

Consider the $C^*$-algebra $\fR^{2n} \cong \Sigma^\alpha C(\S^{2n - 1})$, in which $z_1, \ldots, z_n$ commute with each other. Because the antipodal map $\mathfrak{a}$ is implemented as $\alpha \widehat{\alpha}$ pointwise on $C(\S^{2n - 1}) \rtimes_\alpha \Z_2$, an element or matrix $F(t) + G(t)\delta$ over $\fR^{2n}_\rho$ is odd if and only if each $F(t)$ is odd in $C(\S^{2n - 1})$ and each $G(t)$ is even in $C(\S^{2n - 1})$, which implies that $G(t)$ commutes with $\delta$ and $F(t)$ anticommutes with $\delta$. An odd matrix $F(t) + G(t)\delta$ is then self-adjoint if and only if $F(t)$ and $G(t)$ are self-adjoint. If $F(t) + G(t)\delta$ is also unitary, then it satisfies $(F(t) + G(t)\delta)^2 = 1$, which places tight restrictions on $F$ and $G$.
\be\label{eq:pathoddtostar} \ba
(F(t) + G(t)\delta)^2 	&= F(t)^2 + (G(t)\delta)^2 + F(t)G(t)\delta + G(t)\delta F(t)\\
				&= (F(t)^2 + G(t)^2) + (F(t)G(t) - G(t)F(t))\delta \\
				&= 1 + 0\delta
\ea \ee
This implies that the self-adjoint elements $F(t)$ and $G(t)$ must commute and satisfy $F(t)^2 + G(t)^2 = 1$. In other words, $U(t) = G(t) + iF(t)$ is a unitary matrix over $C(\S^{2n - 1})$ for each $t$, and because $G(t)$ is even and $F(t)$ is odd, $\alpha(U(t)) = U(t)^*$. 

\begin{thm}\label{thm:K0BUcomm-1}
Let $\alpha$ denote the antipodal action on the commutative sphere $C(\S^{2n - 1})$ and consider the path algebra $\fR^{2n}= \Sigma^\alpha C(\S^{2n - 1})$ with antipodal action $\mathfrak{a} = \alpha \widehat{\alpha}$. If $P \in M_{(2m + 1)2^{n}}(\fR^{2n})$ is a projection with $\mathfrak{a}(P) = I - P$, then $[P]_{K_0}$ is not in the subgroup generated by trivial projections.
\end{thm}
\begin{proof}
Write $P = \frac{1}{2}I + \frac{1}{2}B$ where $B$ is an odd, self-adjoint unitary matrix. By the calculation (\ref{eq:pathoddtostar}), $B(t) = F(t) + G(t)\delta$ produces a path of unitaries $U(t) = G(t) + i F(t) \in U_{(2m + 1)2^n}(C(\S^{2n - 1}))$ with $\alpha(U(t)) = U(t)^*$. Since $B(1)$ is a matrix over $\C + \C\delta$, but $F(1)$ must be odd, it follows that $F(1) = 0$ and $B(1) = G(1)\delta$ where $G(1)$ is a self-adjoint matrix with constant entries. Moreover, the boundary conditions of $\Sigma^\alpha C(\S^{2n - 1})$ imply that $G(0) = 0$, so $U(0) = i F(0)$ is an anti self-adjoint odd unitary. If $[P]_{K_0}$ is in the subgroup generated by trivial projections, then the ranks $\textrm{Rank}_{\delta = \pm 1}(P(1))$ are equal, so $U(1) = G(1)$ is a self-adjoint, unitary matrix over $\C$ whose eigenspaces for eigenvalues $\pm 1$ are of equal dimension. It follows that $U(1)$ may be connected within the self-adjoint, unitary matrices over $\C$ to $I_{(2m + 1)2^{n - 1}} \oplus -I_{(2m + 1)2^{n - 1}}$, and the matrices in this path satisfy $\alpha(M) = M = M^*$.

Because $U(0)$ is an anti self-adjoint, odd, unitary matrix, by Example \ref{exam:antioddpaththing}, $U(0)$ is connected via a path of unitaries satisfying $\alpha(M) = M^*$ to a matrix $W$ with identity on the equator and $[W^+]_{K_1} \cong 0 \textrm{ mod } 2$. Similarly, since $U(1)$ is connected via a path of unitaries satisfying $\alpha(M) = M^*$ to $I_{(2m + 1)2^{n - 1}} \oplus -I_{(2m + 1)2^{n - 1}}$, repeated use of Example \ref{exam:identitypaththing} for $(2m + 1)$ summands of $-I_{2^{n - 1}}$ shows $U(1)$ is also connected to a matrix $V$ such that $V$ assigns the identity on the equator and $[V^+]_{K_1} \cong (2m + 1) \cong 1 \textrm{ mod } 2$. This contradicts Theorem \ref{thm:unexpectedZ2}, as $V$ and $W$ assign the identity on the equator and are connected via a path of unitaries satisfying $\alpha(M) = M^*$ (with no assumption on the path's equator data), even though their invariants $[V^+]_{K_1} \textrm{ mod } 2$ and $[W^+]_{K_1} \textrm{ mod } 2$ are different.
\end{proof}

All that remains is to remove the assumption that $z_1, \ldots, z_n$ commute with each other. First, note that the expansion map $E_\alpha: C(\S^{2n - 1}_\omega) \rtimes_\alpha \Z_2 \to M_2(C(\S^{2n - 1}_\omega))$ shows that the relation algebra $\fR^{2n}_\omega = \Sigma^\alpha C(\S^{2n - 1}_\omega)$ embeds into $M_2(C(\S^{2n}_\omega))$, as follows.
\beu \ba
\Sigma^\alpha C(\S^{2n - 1}_\omega) 	&= \{f \in C([0,1], C(\S^{2n - 1}_\omega) \rtimes_\alpha \Z_2): f(0) \in C(\S^{2n - 1}_\omega), f(1) \in \C + \C \delta\} \\
						&\cong \{f \in C([-1,1], C(\S^{2n - 1}_\omega) \rtimes_\alpha \Z_2): f(-1) \in C(\S^{2n - 1}_\omega), f(1) \in \C + \C \delta\} \\
						&\cong \{f \in C([-1, 1], M_2(C(\S^{2n - 1}_\omega))): f(t) = \left[\begin{array}{cc} g(t) & h(t) \\ \alpha(h(t)) & \alpha(g(t))  \end{array} \right] \textrm{ for all } t, \\
						&\phantom{\cong \{ f} h(-1) = 0, \textrm{and } g(1), h(1) \in \C\} \\
						&\leq  M_2(\Sigma C(\S^{2n - 1}_\omega)) = M_2(C(\S^{2n}_\omega))
\ea \eeu

Denote this subalgebra of $M_2(C(\S^{2n}_\omega))$ isomorphic to $\fR^{2n}_\omega$ as $B_\omega$. The antipodal action $\mathfrak{a}$ on $B_\omega$ is realized as $\left[\begin{array}{cc} g(t) & h(t) \\ \alpha(h(t)) & \alpha(g(t))  \end{array} \right] \mapsto \left[\begin{array}{cc} \alpha(g(t)) & -\alpha(h(t)) \\ -h(t) & g(t) \end{array} \right]$, which is entrywise application of $\alpha$ and conjugation by $\left[\begin{array}{cc} 1 & 0 \\ 0 & -1  \end{array} \right]$. The spheres $C(\S^{2n}_\omega)$ are $\theta$-deformations of $C(\S^{2n})$, so if $f$ and $g$ are fixed (matrices of) smooth elements, the following continuity properties apply, as weak forms of strong quantization (\cite{ri93}). Here $\mathcal{C_\omega}$ denotes the collection of parameter matrices $\rho$ which differ from $\omega$ only in one prescribed pair of conjugate entries, $(i, j)$ and $(j, i)$.
\beu
||f \cdot_\omega g - f \cdot_\rho g||_\rho \to 0 \textrm{ as } \rho \to \omega \textrm{ within } \mathcal{C}_\omega
\eeu
\beu
||f||_\rho \to ||f||_\omega \textrm{ as }  \rho \to \omega \textrm{ within } \mathcal{C}_\omega
\eeu

Any smooth approximations in $M_2(C(\S^{2n}_\omega))$ to an element of $B_\omega$ can be made without leaving $B_\omega$. Note that the pointwise action $\alpha$ on $\Sigma C(\S^{2n - 1}_\omega) = C(\S^{2n}_\omega)$ is an action $\gamma$ on $C(\S^{2n}_\omega)$ which fixes $x$ but negates every generator $z_i$, so $\gamma$ commutes with the rotation action of $\R^n$ defining the Rieffel deformation. Moreover, because the antipodal map on $B_\omega$ is implemented with an entrywise action and a conjugation, a matrix function $\left[\begin{array}{cc} g(t) & h(t) \\ \alpha(h(t)) & \alpha(g(t))  \end{array} \right] = \left[\begin{array}{cc} g & h \\ \gamma(h) & \gamma(g) \end{array}\right]$ in $B_\omega$ is odd if and only if $\gamma(g) = -g$ and $\gamma(h) = h$. Together, these observations imply that smooth approximations to $B_\omega$ elements may be made in $B_\omega$ while preserving homogeneity properties in the antipodal action. Also, as the adjoint operations of Rieffel deformations all have the same effect on smooth elements, smooth approximations can also be made to preserve self-adjointness.

\begin{thm}
Suppose $P \in M_{(2m + 1)2^n}(\fR^{2n}_\omega)$ is a projection with $\mathfrak{a}(P) = I - P$. Then $[P]_{K_0}$ is not in the subgroup generated by the trivial projections.
\end{thm}
\begin{proof}
Suppose the theorem fails for $P \in M_{(2m + 1)2^n}(\fR^{2n}_\omega)$, so there is a path of projections connecting $P \oplus 0 \oplus I$ to a trivial projection $I \oplus 0$. If this path is viewed under the isomorphism $\fR^{2n}_\omega \cong B_\omega \leq M_2(C(\S^{2n}_\omega))$, then any matrix over $B_\omega$ can be approximated by matrices with smooth entries, and this approximation can be done without leaving $B_\omega$. Moreover, if the original matrix is odd or self-adjoint, these properties can be preserved in the approximation. Any smooth matrix over $B_\omega \leq M_2(C(\S^{2n - 1}_\omega))$ may then be viewed as a matrix over $B_\rho \leq M_2(C(\S^{2n - 1}_\rho))$ for any $\rho$.

The path of projections connecting $P \oplus 0 \oplus I$ to $I \oplus 0$ is uniformly continuous, so let $P_0, \ldots, P_k$ be finitely many elements of this path with the following properties.
\beu
P_0 = P \oplus 0 \oplus I \hspace{1 cm} ||P_j - P_{j - 1}||_\omega < \varepsilon \hspace{1 cm} P_k = I  \oplus 0
\eeu
Fix $0 < \varepsilon < \frac{1}{13}$, choose a self-adjoint smooth approximation $Q$ to $P$ such that $\mathfrak{a}(Q) = I - Q$ and $||Q \cdot_\omega Q - Q||_\omega < \varepsilon$, and let $Q_0 = Q \oplus 0 \oplus I$. It follows that $||Q_0 \cdot_\omega Q_0 - Q_0||_\omega < \varepsilon$. The matrix $P_k = I \oplus 0$ is already smooth, so let $Q_k = P_k$. Finally, choose smooth approximations $Q_1, \ldots, Q_{k - 1}$ of $P_1, \ldots, P_{k - 1}$ so that these restrictions hold for each $j$.
\beu
Q_0 = Q \oplus 0 \oplus I \hspace{.6 cm} Q_j = Q_j^* \hspace{.6 cm} ||Q_j \cdot_\omega Q_j - Q_j||_\omega < \varepsilon \hspace{.6 cm} ||Q_j - Q_{j - 1}||_\omega < \varepsilon \hspace{.6 cm} Q_k = I \oplus 0
\eeu

If $\rho$ is another parameter matrix differing from $\omega$ only in a prescribed pair of conjugate entries, then if $\rho$ is close enough to $\omega$, similar inequalities hold.
\beu
Q_0 = Q \oplus 0 \oplus I \hspace{.6 cm} Q_j = Q_j^* \hspace{.6 cm} ||Q_j \cdot_\rho Q_j - Q_j||_\rho < \varepsilon \hspace{.6 cm} ||Q_j - Q_{j - 1}||_\rho < \varepsilon \hspace{.6 cm} Q_k = I \oplus 0
\eeu
Let $R_t$, $t \in [0, 1]$, denote the piecewise linear path connecting $Q_0, \ldots, Q_k$, so each $R_t$ is self-adjoint and has $||R_t||_\rho \leq \textrm{max}\hspace{3 pt}\{||Q_j||_\rho: j \in \{0, \ldots, k\} \} \leq 1 + \varepsilon $. Further, for any $t$ there is a $j$ such that $||R_t - Q_j||_{\rho} < \varepsilon$, so liberal use of the triangle inequality and properties of $Q_j$ shows that $||R_t \cdot_\rho R_t - R_t||_\rho < 3\varepsilon + 2 \varepsilon^2$.
\beu
R_0 = Q \oplus 0 \oplus I \hspace{1 cm} R_t = R_t^* \hspace{1 cm} ||R_t \cdot_\rho R_t - R_t||_\rho < 3\varepsilon + 2 \varepsilon^2 \hspace{1 cm} R_1 = I \oplus 0
\eeu

Write $R_t = \frac{1}{2}I + \frac{1}{2}F_t$, so $F_t$ is self-adjoint and $||F_t \cdot_\rho F_t - 1||_\rho < 12\varepsilon + 8\varepsilon^2 < 1$. This implies that $F_t$ is invertible, so normalize $F_t$ to a self-adjoint unitary (under $\cdot_\rho$) to produce $F_t \cdot_\rho |F_t|^{-1_\rho}$ and the projection $\widetilde{R_t} = \frac{1}{2}I + \frac{1}{2}F_t \cdot_\rho |F_t|^{-1_\rho}$. Note that this normalization process does nothing to $R_1$ and respects the summands of $R_0$: $\widetilde{R_0} = \widetilde{Q} \oplus 0 \oplus I$ where $\widetilde{Q}$ is a projection obtained $Q$ in the $(2m + 1)2^{n}$-dimensional matrix algebra by similar means. Moreover, since $Q = \frac{1}{2}I + \frac{1}{2}F$ satisfies $\mathfrak{a}(Q) = I - Q$, $F$ is a self-adjoint odd invertible, so $F \cdot_\rho |F|^{-1_\rho}$ is a self-adjoint odd unitary, and $\mathfrak{a}(\widetilde{Q}) = I - \widetilde{Q}$. Finally, the path $\widetilde{R_t}$ is a path of projections in $B_\rho \cong \fR^{2n}_\rho$ connecting $\widetilde{Q} \oplus 0 \oplus I$ to $I \oplus 0$, and $\widetilde{Q} \in M_{(2m+1)2^{n}}(B_\rho) = M_{(2m+1)2^n}(\fR^{2n}_\rho)$ is in the $K_0$ subgroup generated by trivial projections. This means that assuming the theorem fails for a single parameter matrix $\omega$ implies that the theorem fails for all $\rho$ that are sufficiently close to $\omega$ and differ from $\rho$ in a prescribed pair of conjugate entries. We may select $\rho$ such that this pair of conjugate entries consists of roots of unity of odd order, and then repeat the argument starting with $\rho$ and another pair of conjugate entries. In finitely many iterations, it will follow that the theorem fails for a parameter matrix $\mu$ which has odd order roots of unity in every entry. This is a contradiction by the following argument.

Suppose $\mu$ is a parameter matrix with an odd order root of unity in each entry such that the theorem fails for a projection $P \in M_{(2m+1)2^n}(\fR^{2n}_\mu)$, and consider $\fR^{2n} = \Sigma^\alpha C(\S^{2n - 1})$. There are unitary matrices $V_1, \ldots, V_n \in U_{2q+1}(\C)$ such that $V_k V_j = \mu_{jk} V_j V_k$, so define $B: \fR^{2n}_\mu \mapsto M_{2q + 1}(\fR^{2n})$ as follows.
\beu
B: z_j \mapsto z_j V_j \hspace{1 in} B: x \mapsto x I_{2q+1}
\eeu
The $*$-homomorphism $B$ exists because the desired images satisfy the relations defining $\fR^{2n}_\mu$, as all noncommutativity information among $z_1, \ldots, z_n$ is pushed to the matrices $V_j$. Further, $B$ is equivariant for the antipodal map, as the odd generators are sent to matrices with odd entries. So, if $P \in M_{(2m+1)2^n}(\fR^{2n}_\mu)$ satisfies $\mathfrak{a}(P) = I - P$ and $[P]_{K_0}$ is in the subgroup generated by trivial projections, then the same properties apply to $B(P) \in M_{(2q + 1)(2m + 1)2^n}(\fR^{2n})$, contradicting Theorem \ref{thm:K0BUcomm-1}.
\end{proof}

\begin{cor}
Suppose $\Phi: \fR^{2n}_\omega \to \fR^{2n}_\rho$ is a unital $*$-homomorphism that is equivariant for the antipodal map. If $K_0 \cong \Z \oplus \Z$ is such that the first summand is generated by trivial projections, then the $K_0$ map induced by $\Phi$ is nontrivial on the second summand.
\end{cor}
\begin{proof}
The image $\Phi(P^\prime_\omega(n))$, where $P^\prime_\omega(n)$ is as in (\ref{eq:distinguishedprojection}), is a $2^n \times 2^n$ projection that satisfies $\alpha(\Phi(P^\prime_\omega(n))) = I - \Phi(P^\prime_\omega(n))$. Its $K_0$ class is therefore not in the subgroup generated by trivial projections.
\end{proof}
\noindent \textit{Remark.} Because the $\theta$-deformed spheres have the same nontriviality statement in $K_0$, the domain or codomain could be replaced by a $\theta$-deformed sphere of dimension $2n$.

A curious aspect of these proofs is that they discuss unitaries satisfying $\alpha(M) = M^*$ over $C(\S^{2n - 1}_\omega)$, a sphere one dimension lower than the original, whereas the argument for $C(\S^{2n}_\omega)$ worked by pushing up one dimension and focusing on odd unitaries. In both cases, the alternative method of proof appears to have \lq\lq insurmountable\rq\rq \hspace{0 pt} road blocks, where one switch of sign nullifies an entire argument. Nevertheless, we have seen that a noncommutative Borsuk-Ulam theorem does hold in both cases.

\section{Acknowledgments}
This work was partially supported by NSF grants DMS 1300280 and DMS 1363250. Moreover, I am grateful to my advisors John McCarthy and Xiang Tang at Washington University in St. Louis for their continued support.

\bibliography{anti}
\end{document}